\documentclass[11pt,reqno]{amsart}
\usepackage{amsmath}
\usepackage{amsthm}
\usepackage{amsfonts}
\usepackage{amssymb}
\usepackage{bm}
\usepackage{tikz}
\usepackage{float}
\usepackage[all]{xy}
\usepackage{hyperref}
\usepackage{enumerate}
\usepackage{enumitem}
\usepackage[english]{babel}
\usepackage{mathtools}

\newcommand{\op}{\ensuremath{^{\mathrm{op}}}}

\newcommand{\bA}{\mathbb{A}}

\newcommand{\bZ}{\mathbb{Z}}

\newcommand{\cC}{{\mathcal C}}

\newcommand{\cE}{{\mathcal E}}
\newcommand{\cF}{{\mathcal F}}
\newcommand{\cG}{{\mathcal G}}

\newcommand{\cM}{{\mathcal M}}

\newcommand{\cP}{{\mathcal P}}

\newcommand{\cS}{{\mathcal S}}

\newcommand{\cU}{{\mathcal U}}

\newcommand{\cX}{{\mathcal X}}
\newcommand{\cY}{{\mathcal Y}}

\newcommand{\md}{\operatorname{mod}}
\newcommand{\Md}{\operatorname{Mod}}

\newcommand{\perf}{\operatorname{perf}}

\newcommand{\emphbf}[1]{\emph{\textbf{#1}}}
\newcommand{\Hom}{\operatorname{Hom}}

\newcommand{\Ext}{\operatorname{Ext}}

\newcommand{\Ker}{\operatorname{Ker}}
\newcommand{\Coker}{\operatorname{Coker}}

\newcommand{\colim}{\operatorname{colim}}

\usetikzlibrary{matrix,arrows}
\title{$d\mathbb{Z}$-cluster tilting subcategories of singularity categories}

\date{\today}


\author{Sondre Kvamme}
\address{Sondre Kvamme\\
Department of Mathematics, Uppsala University 
\\ 75106 Uppsala, Sweden
} \email{sondre.kvamme@math.uu.se}

\begin{document}

\newtheorem{Theorem}[equation]{Theorem}
\newtheorem{Lemma}[equation]{Lemma}
\newtheorem{Corollary}[equation]{Corollary}
\newtheorem{Proposition}[equation]{Proposition}
\newtheorem{Conjecture}[equation]{Conjecture}

\theoremstyle{definition}
\newtheorem{Definition}[equation]{Definition}
\newtheorem{Example}[equation]{Example}
\newtheorem{Remark}[equation]{Remark}
\newtheorem{Setting}[equation]{Setting}

\thanks{The author was supported by a public grant as part of the FMJH. He would like to thank the anonymous referee for several useful comments and corrections}


\begin{abstract}
For an exact category $\cE$ with enough projectives and with a $d\bZ$-cluster tilting subcategory, we show that the singularity category of $\cE$ admits a $d\bZ$-cluster tilting subcategory. To do this we introduce cluster tilting subcategories of left triangulated categories, and we show that there is a correspondence between cluster tilting subcategories of $\cE$ and $\underline{\cE}$. We also deduce that the Gorenstein projectives of $\cE$ admit a $d\bZ$-cluster tilting subcategory under some assumptions. Finally, we compute the $d\bZ$-cluster tilting subcategory of the singularity category for a finite-dimensional algebra which is not Iwanaga-Gorenstein.
\end{abstract}

\maketitle

\setcounter{tocdepth}{2}
\numberwithin{equation}{section}
\tableofcontents

\section{Introduction}
Auslander-Reiten theory is a fundamental tool to describe the module category of finite-dimensional algebras, see \cite{ARS95} and \cite{ASS06,SS07a,SS07}. A generalization of this theory, called higher Auslander-Reiten theory, was introduced by Iyama in \cite{Iya07a} and further developed in \cite{Iya07,Iya11}. In this case, the objects of study are module categories equipped with a $d$-cluster tilting subcategory. We refer to \cite{AO14,GKO13,HI11a,HI11,HIO14,IJ17,IO11,IO13,IY08,Jas16,JK19,J16} for some other important papers. Also, see \cite{Iya08} for a survey of the theory and \cite{JK19a} for an introduction.

Let $\Lambda$ be a finite-dimensional algebra and let $\md \Lambda$ be the category of finitely generated right $\Lambda$-modules. Assume $\Lambda$ has global dimension $d$. If $\cM$ is a $d$-cluster tilting subcategory in $\md \Lambda$ then the subcategory 
\[
\cU= \operatorname{add}\{M[di]\in D^b(\md \Lambda) \mid M\in \cM \text{ and }i\in \bZ\}
\]
is $d$-cluster tilting inside the bounded derived category $D^b(\md \Lambda)$ \cite[Theorem 1.21]{Iya11} (this can be extended to $\tau_d$-finite algebras \cite[Theorem 1.23]{Iya11} and cluster categories \cite[Theorem 4.10]{Ami09} and \cite[Theorem 2.2]{Guo11}). The subcategory $\cU$ can be considered as a higher analogue of the derived category of a hereditary algebra. On the other hand, if $\Lambda$ does not have global dimension $d$, then there is no known cluster tilting subcategory inside $D^b(\md \Lambda)$ in general. As shown in \cite{JK16}, the naive approach doesn't necessarily give a cluster tilting subcategory even when $\cM$ is $d\bZ$-cluster tilting.

In this paper we consider instead the singularity category
\[
D_{\operatorname{sing}}(\Lambda):=K^{-,b}(\operatorname{proj} \Lambda)/K^b(\operatorname{proj} \Lambda)
\]
where $K^b(\operatorname{proj} \Lambda)$ and $K^{-,b}(\operatorname{proj} \Lambda)$ denote the bounded homotopy category of finitely generated projective modules and the right bounded homotopy category with bounded homology of finitely generated projective modules, respectively. The singularity category was introduced by Buchweitz in \cite{Buc86} as an useful invariant of the ring $\Lambda$. Via the equivalence $K^{-,b}(\operatorname{proj} \Lambda)\cong D^b(\md \Lambda)$ we get an equivalence
\begin{equation}\label{how to write singularity category}
K^{-,b}(\operatorname{proj} \Lambda)/K^b(\operatorname{proj} \Lambda)\cong D^b(\md \Lambda)/\operatorname{perf} \Lambda
\end{equation}
and hence $D_{\operatorname{sing}}(\Lambda)$ can be realized as a quotient of $D^b(\md \Lambda)$. We show that if $\md \Lambda$ has a $d\bZ$-cluster tilting subcategory, then $D_{\operatorname{sing}}(\Lambda)$ has a $d\bZ$-cluster tilting subcategory.  In fact, we show this more generally for any exact category with enough projectives.

\begin{Theorem}\label{Main theorem intro}
Let $\cE$ be an exact category with enough projectives $\cP$ and with a $d\bZ$-cluster tilting subcategory $\cM$. Then the subcategory 
\[
\{P^{\bullet}\in K^{-,b}(\cP)/K^b(\cP) \mid Z^{di}(P^{\bullet})\in \cM \text{ for all } i\ll 0\}
\]
is a $d\bZ$-cluster tilting subcategory of $K^{-,b}(\cP)/K^b(\cP)$.
\end{Theorem}
Via the equivalence \eqref{how to write singularity category} this corresponds to the subcategory
\[
\{C^{\bullet}\in D^b(\md \Lambda)/\operatorname{perf} \Lambda \mid C^{\bullet}\cong M[di] \text{ for some } M\in \cM \text{ and } i\in \bZ\}
\]
in $D^b(\md \Lambda)/\operatorname{perf}\Lambda$. Notice that this subcategory is closed under direct sums since any object in $D^b(\md \Lambda)/\operatorname{perf}\Lambda$ is isomorphic to a stalk complex. There are many examples of $d$-cluster tilting subcategories of singularity categories of Iwanaga-Gorenstein rings, see \cite{Iya18}, but Theorem \ref{Main theorem intro} is the first result for non-Iwanaga-Gorenstein rings. For such rings the singularity category is more difficult to control since it is not enhanced by the Gorenstein projective modules. However, it is still possible to compute the subcategory explicitly, which we do in Example \ref{Example 2}. 

To prove Theorem \ref{Main theorem intro}, we use the left triangulated structure of $\underline{\cE}$, see Definition \ref{Left triangulated categories} and Theorem \ref{stable category is left triangulated}. More precisely, we introduce $d$-cluster tilting subcategories of left triangulated categories, and we show that there is a one-to-one correspondence between $d$ and $d\bZ$-cluster tilting subcategories of $\cE$ and $\underline{\cE}$, see Theorem \ref{correspondence cluster tilting exact and stable}. Furthermore, we show that if a left triangulated category has a $d\bZ$-cluster tilting subcategory, then its stabilization has a $d\bZ$-cluster tilting subcategory. We then conclude using the fact that the stabilization of $\underline{\cE}$ is the singularity category, which was proved in \cite{KV87}.

We also obtain a corollary for Gorenstein projective modules, which we state below in the special case where $\cE=\md \Lambda$. Recall that $\Lambda$ is Iwanaga-Gorenstein if it has finite selfinjective dimension, and that in this case the Gorenstein projectives are 
\[
\cG\cP(\md \Lambda) := \{M\in \md \Lambda \mid \Ext^{i}_{\Lambda}(M,\Lambda)=0 \text{ for  }i>0\}.
\]

\begin{Corollary}[Theorem \ref{Cluster tilting inside Gorenstein projectives}]
Assume $\Lambda$ is Iwanaga-Gorenstein, and let $\cM$ be a $d\bZ$-cluster tilting subcategory of $\md \Lambda$. Then 
\[
\cM\cap \cG\cP(\md \Lambda)
\]
is a $d\bZ$-cluster tilting subcategory of $\cG\cP(\md \Lambda)$.
\end{Corollary}

We now describe the structure of the paper. In Section \ref{Section Exact categories}, \ref{Section Left triangulated categories} and \ref{Section Cluster-tilting subcategories} we recall the essential notions and results which we need. In Section \ref{Section Cluster tilting subcategories of left triangulated categories} we introduce cluster tilting subcategories for left triangulated categories. We show that for a left triangulated category $\cC$ with a $d\bZ$-cluster tilting subcategory, the stabilization $\bZ\cC$ has a $d\bZ$-cluster tilting subcategory. We also investigate the $(d+2)$-angulated structure of this subcategory. Our main result in Section \ref{Section d-cluster tilting in stable categories} is Theorem \ref{correspondence cluster tilting exact and stable}, which gives a correspondence between $d$-cluster tilting subcategories of $\cE$ and $\underline{\cE}$ when $\cE$ is an exact category with enough projectives. In Section \ref{Section Gorenstein projectives} we investigate the relationship with Gorenstein projectives. In Section \ref{Section Example} we compute the cluster tilting subcategory of the singularity category in two examples.

\section{Exact categories}\label{Section Exact categories}
Here we define exact categories, following the conventions in \cite{Bue10}. 

\begin{Definition} An \emphbf{exact category} $\cE$ is an additive category equipped with a distinguished class of sequences
\[
0\to E_3\xrightarrow{f} E_2\xrightarrow{g}E_1\to 0
\]
where $f$ is the kernel of $g$ and $g$ is the cokernel of $f$. The morphisms $f$ are called \emphbf{admissible monomorphisms}, and the morphisms $g$ are called \emphbf{admissible epimorphisms}. The following axioms need to be satisfied:
\begin{enumerate}
\item[(E0)] For all object $E$ in $\cE$ the identity morphism $1_E\colon E\to E$ is an admissible monomorphism;
\item[(E0$\op$)]  For all object $E$ in $\cE$ the identity morphism $1_E\colon E\to E$ is an admissible epimorphism;
\item[(E1)] The composite of two admissible monomorphism is an admissible monomorphism;
\item[(E1$\op$)] The composite of two admissible epimorphisms is an admissible epimorphism;
\item[(E2)] The pushout of an admissible monomorphism exists and yields an admissible monomorphisms. In other words, given an admissible monomorphism $f\colon E_0\to E_1$ and a morphism $g\colon E_0\to E_2$ there exists a pushout diagram
\[
\xymatrix{
 E_0\ar[d]^{g}\ar[r]^{f} & E_1\ar[d]^{h}\\
 E_2 \ar[r]^{k} & E_3
}
\]
where $k$ is an admissible monomorphism;
\item[(E2$\op$)] The pullback of an admissible epimorphism exists and yields an admissible epimorphism. In other words, given an admissible epimorphism $f\colon E_1\to E_0$ and a morphism $g\colon E_2\to E_0$ there exists a pullback diagram
\[
\xymatrix{
 E_3\ar[d]^{k}\ar[r]^{h} & E_2\ar[d]^{g}\\
 E_1 \ar[r]^{f} & E_0
}
\]
where $h$ is an admissible epimorphism.
\end{enumerate}
\end{Definition}
A sequence in $\cE$
\[
\cdots \to E_1\to E_0\to \cdots
\]
is called exact if there exist exact sequences $0\to K_i\to E_i\to K_{i-1}\to 0$ such that the maps $E_i\to E_{i-1}$ are equal to the composites $E_i\to K_{i-1}\to E_{i-1}$. If $\cE$ is an exact category, then the opposite category $\cE\op$ becomes an exact category in a natural way. If $\cF\subset \cE$ is a full subcategory of $\cE$ which is closed under extensions, then the class of sequences $0\to F_1\xrightarrow{} F_2\xrightarrow{}F_3\to 0$ in $\cF$  which are exact in $\cE$ makes $\cF$ into an exact category. We say that $\cF$ is an exact subcategory of $\cE$. 

An object $P$ in $\cE$ is projective if for any admissible epimorphism $E_1\to E_0$ the induced map $\cE(P,E_1)\to \cE(P,E_0)$ is an epimorphism. We let $\cP$ denote the subcategory of $\cE$ consisting of the projective objects. The exact category $\cE$ has enough projectives if for any object $E$ in $\cE$ there exists an admissible epimorphism $P\to E$ with $P\in \cP$. In this case we let $\underline{\cE}=\cE/\cP$ denote the stable category of $\cE$ modulo projectives. For any object $E$ or morphism $f$ in $\cE$ we denote the corresponding object or morphism in $\underline{\cE}$ by $\underline{E}$ or $\underline{f}$. Since $\cE$ has enough projectives, $\underline{f}=0$ for a morphism $f\colon E\to E'$ if and only if $f$ factors through an admissible epimorphism $P\to E'$ with $P$ projective. Using this, it follows by the same argument as in the proof of \cite[Theorem 2.2]{Hel60} that for any two objects $E_0$ and $E_1$ in $\cE$ there exists an isomorphism $\underline{E_1}\cong \underline{E_2}$ in $\underline{\cE}$ if and only if there exist projective objects $P,Q\in \cE$ and an isomorphism
\begin{equation}\label{isomorphism in stable cat}
E_0\oplus P\cong E_1\oplus Q
\end{equation}
in $\cE$.

Next we define the syzygy functor:

\begin{Definition}[\cite{Hel60}]\label{syzygy functor}
Let $\cE$ be an exact category with enough projectives. For each object $E\in \cE$ choose an admissible epimorphism $p_E\colon P_E\to E$ with $P_E$ projective. The \emphbf{syzygy functor} $\Omega\colon \underline{\cE}\to \underline{\cE}$ is defined as follows: For an object $E\in \cE$ we set $\Omega \underline{E}= \underline{\Ker p_E}$. For a morphism $f\colon E_0\to E_1$ we set $\Omega(\underline{f})=\underline{h}$ where $g\colon P_{E_0}\to P_{E_1}$ and $h\colon \Ker p_{E_0}\to \Ker p_{E_1}$ are morphisms in $\cE$ making the diagram
\[
\xymatrix{
0\ar[r]& \Ker p_{E_1} \ar[r]^{} & P_{E_1}\ar[r]^{p_{E_1}} &
E_1 \ar[r] & 0  \\
0\ar[r] & \Ker p_{E_0} \ar[r]\ar[u]^{h} & P_{E_0}\ar[r]^{p_{E_0}}\ar[u]^g & E_0\ar[r]\ar[u]^{f} & 0
}
\]
commutative.
\end{Definition} 
Note that $\Omega(\underline{f})$ is independent of the choice of morphisms $g$ and $h$, and up to isomorphism the syzygy functor is independent of the choice of the admissible epimorphisms $p_E$, see Section 3 in \cite{Hel60} for details. By abuse of notation, for $E\in \cE$ and $i\geq 0$ we let $\Omega^iE$ denote a choice of an object in $\cE$ satisfying $\underline{\Omega^iE}\cong \Omega^i\underline{E}$.

An object $I$ in $\cE$ is injective if it is projective in $\cE\op$, and $\cE$ has enough injectives if $\cE\op$ has enough projectives. The exact category $\cE$ is called \emphbf{Frobenius} if $\cE$ has enough projectives and enough injectives, and if the projective and injective objects coincide. In this case, the stable category $\underline{\cE}$ becomes a triangulated category where the suspension functor is the quasi-inverse of $\Omega$. We refer to \cite[Section I.2]{Hap88} for more details.

We end this section with the following lemma, which we need later.

\begin{Lemma}\label{Syzygy functor full and faithful}
Let $\cE$ be an exact category with enough projectives, and let $E$ be an object in $\cE$. The following statements are equivalent:
\begin{enumerate}
\item\label{syzygy full:1} $\Ext^1_{\cE}(E,P)=0$ for all $P\in \cP$;
\item\label{syzygy full:2} The map 
\[
\Hom_{\underline{\cE}}(\underline{E},\underline{E'}) \to \Hom_{\underline{\cE}}(\Omega (\underline{E}), \Omega (\underline{E'})) \quad \underline{f}\mapsto \Omega(\underline{f})
\]
 is an isomorphism for all $E'\in \cE$.
\end{enumerate}
\end{Lemma} 

\begin{proof}
This is well known, see \cite[Proposition 2.43]{AB69}. See also \cite[Lemma 9]{KM19} for a direct proof which can be translated verbatim to an exact category.
\end{proof}

\section{Left triangulated categories}\label{Section Left triangulated categories}
Here we recall the notion of a left triangulated category. This was first considered in  \cite{Kel90}, \cite{KV87} (where it would be called a co-suspended category), and later in \cite{ABM98}, \cite{Bel00}, \cite{BM94}. A higher dimensional version has also been introduced in \cite{Lin17}.

Let $\cC$ be a category and $\Omega\colon \cC\to \cC$ an endofunctor. A sequence of the form 
\[
\Omega C\to A\to B\to C
\]
is called an $\Omega$-sequence. A morphism of $\Omega$-sequences $\Omega C\to A\to B\to C$ and $\Omega C'\to A'\to B'\to C'$ is a commutative diagram
\[
\xymatrix{
\Omega C\ar[d]^{\Omega h}\ar[r]& A \ar[d]^f\ar[r]^{} & B \ar[d]^g\ar[r] & C\ar[d]^h \\
\Omega C'\ar[r]& A' \ar[r]^{} & B'\ar[r] & C .
}
\]
Composition of morphism is given in the canonical way.

\begin{Definition}\label{Left triangulated categories}
A \emphbf{left triangulated category} is an additive category $\cC$ equipped with an additive endofunctor $\Omega\colon \cC\to \cC$ and a class of $\Omega$-sequences called \emphbf{triangles}, which satisfy the following axioms:
\begin{enumerate}
\item[(T0)]\label{Left triangulated categories:0} Any $\Omega$-sequence which is isomorphic to a triangle is a triangle itself;
\item[(T1)]\label{Left triangulated categories:1} For each $C\in \cC$ the $\Omega$-sequence $0\to C\xrightarrow{1}C\to 0$ is a triangle;
\item[(T2)]\label{Left triangulated categories:2} For any morphism $w\colon B\to C$ there exists a triangle $\Omega C\to A\to B\xrightarrow{w} C$
\item[(T3)]\label{Left triangulated categories:3} If $\Omega C\xrightarrow{u} A\xrightarrow{v} B\xrightarrow{w} C$ is a triangle, then $\Omega B\xrightarrow{-\Omega (w)} \Omega C\xrightarrow{u} A\xrightarrow{v} B$ is a triangle;
\item[(T4)]\label{Left triangulated categories:4} Given a diagram
\[
\xymatrix{
\Omega C\ar[d]^{\Omega h}\ar[r]^u & A \ar[r]^{v} & B \ar[d]^g\ar[r]^w & C\ar[d]^h \\
\Omega C'\ar[r]^{u'}& A' \ar[r]^{v'} & B'\ar[r]^{w'} & C
}
\]
where the rows are triangles and the square commutes, then there exists a morphism $f\colon A\to A'$ making the whole diagram commute;
\item[(T5)]\label{Left triangulated categories:5} Given two morphisms $w\colon B\to C$ and $h\colon C\to D$, triangles $\Omega C\xrightarrow{u} A\xrightarrow{v} B\xrightarrow{w} C$ and $\Omega D\xrightarrow{f} E\xrightarrow{g} C\xrightarrow{h} D$, then there exists a commutative diagram
\[
\xymatrix{
& \Omega E \ar[d]^{u\circ \Omega (g)} & & \\
\Omega C\ar[d]^{\Omega(h)}\ar[r]^u & A \ar[r]^{v}\ar[d]^{\alpha} & B \ar[d]^1\ar[r]^w & C\ar[d]^h \\
\Omega D\ar[d]^1\ar[r]^{i}& F \ar[d]^{\beta}\ar[r]^{j} & B\ar[d]^w \ar[r]^{h\circ w} & D \ar[d]^1 \\
\Omega D\ar[r]^{f}& E \ar[r]^{g} & C\ar[r]^{h} & D 
}
\]
where the middle row and the second column are also triangles.
\end{enumerate}
\end{Definition}
If $\Omega C\to A\to B\xrightarrow{} C$ is a triangle and $D\in \cC$, then applying $\Hom_{\cC}(D,-)$ gives a long exact sequence
\begin{multline*}
\cdots \to \Hom_{\cC}(D,\Omega^{i+1} C)\to  \Hom_{\cC}(D,\Omega^i A)\to \Hom_{\cC}(D,\Omega^i B)\xrightarrow{} \cdots  \\
\cdots \to \Hom_{\cC}(D,\Omega C)\to  \Hom_{\cC}(D, A)\to \Hom_{\cC}(D,B) \to \Hom_{\cC}(D,C)
\end{multline*}
see \cite[Corollary 1.5]{ABM98}. In particular, if $0\to A\xrightarrow{v}B\to 0$ is a triangle, then $\Hom_{\cC}(D, A)\xrightarrow{v\circ -} \Hom_{\cC}(D,B)$ is an isomorphism for all $D\in \cC$, so $v$ must be an isomorphism. 

Assume $\cE$ is an exact category with enough projectives. Any diagram of the form
\[
\xymatrix{
0\ar[r]& E_3\ar[r]^{v} & E_2\ar[r]^{w} &
E_1 \ar[r] & 0  \\
0\ar[r] & \Omega E_1 \ar[r]\ar[u]^{u} & P\ar[r]\ar[u] & E_1\ar[r]\ar[u]^{1} & 0
}
\]
induces a sequence $\Omega \underline{E_1}\xrightarrow{\underline{u}} \underline{E_3}\xrightarrow{\underline{v}}\underline{E_2}\xrightarrow{\underline{w}}\underline{E_1}$ in $\underline{\cE}$. The following result follows from \cite[Theorem 3.1]{BM94}.

\begin{Theorem}\label{stable category is left triangulated}
Let $\cE$ be an exact category with enough projectives. Then the stable category $\underline{\cE}$ together with syzygy functor $\Omega\colon \underline{\cE}\to \underline{\cE}$ and the class of all $\Omega$-sequences isomorphic to a sequence $\Omega \underline{E_1}\xrightarrow{\underline{u}} \underline{E_3}\xrightarrow{\underline{v}}\underline{E_2}\xrightarrow{\underline{w}}\underline{E_1}$ as above is a left triangulated category.
\end{Theorem}

Now assume $\cC$ is a category and $\Omega\colon \cC\to \cC$ is an endofunctor. Let $\bZ\cC$ be the stabilization of $\cE$ \cite{Hel68}. Explicitly, the objects of $\bZ\cC$ are pairs $(C,n)$ where $C$ is an object in $\cC$ and $n\in \bZ$ is an integer. The morphism space between two objects $(C,m)$ and $(C',n)$ is given by
\[
\Hom_{\bZ\cC}((C,m),(C',n)):= \colim_{k}\Hom_{\cC}(\Omega^{m+k}(C),\Omega^{n+k}(C'))
\]
where the colimit is taken over all $k\in \bZ$ such that $m+k\geq 0$, and $n+k\geq 0$. Since this is a filtered colimit, it follows that any morphism $(C,m)\to (C',n)$ has a representative $\Omega^{m+k}(C)\to \Omega^{n+k}(C')$ in $\cC$ for some $k$. Composition $(C,m)\to (C',n)\to (C'',p)$ in $\bZ\cC$ is given by composing representatives $\Omega^{m+k}(C)\to \Omega^{n+k}(C')\to \Omega^{p+k}(C'')$ in $\cC$. The category $\bZ\cC$ comes equipped with mutual inverse automorphisms
\begin{equation}\label{suspension and loop}
\Sigma\colon \bZ\cC\to \bZ\cC \quad \text{and} \quad \Omega\colon \bZ\cC\to \bZ\cC 
\end{equation}
given by $\Sigma(C,n)=(C,n-1)$ and $\Omega(C,n)=(C,n+1)$ on objects and by the canonical identifications
\begin{align*}
\colim_{k}\Hom_{\cC}(\Omega^{m+k}C,\Omega^{n+k}C')\cong \colim_{k}\Hom_{\cC}(\Omega^{m-1+k}C,\Omega^{n-1+k}C') \\
\colim_{k}\Hom_{\cC}(\Omega^{m+k}C,\Omega^{n+k}C')\cong \colim_{k}\Hom_{\cC}(\Omega^{m+1+k}C,\Omega^{n+1+k}C')
\end{align*}
to get maps
\begin{align*}
\Hom_{\bZ\cC}((C,m),(C',n))\xrightarrow{\cong} \Hom_{\bZ\cC}((C,m-1),(C',n-1)) \quad f\mapsto \Sigma(f) \\
\Hom_{\bZ\cC}((C,m),(C',n))\xrightarrow{\cong} \Hom_{\bZ\cC}((C,m+1),(C',n+1)) \quad f\mapsto \Omega(f)
\end{align*}
on the morphisms spaces.

  If $\cC$ is a left triangulated category, then $\bZ\cC$ also comes equipped with a class of sequences 
\[
\Omega (C_1,n_1)\to (C_3,n_3)\to (C_2,n_2)\to (C_1,n_1)
\]
called \emphbf{standard triangles}, which are induced from a sequence 
\[
\Omega^{k+n_1+1}(C_1)\xrightarrow{u}\Omega^{k+n_3}(C_3)\xrightarrow{v} \Omega^{k+n_2}(C_2)\xrightarrow{w} \Omega^{k+n_1}(C_1)
\]
in $\cC$ for some integer $k\in \bZ$, and where 
\[
\Omega^{k+n_1+1}(C_1)\xrightarrow{(-1)^{k}u}\Omega^{k+n_3}(C_3)\xrightarrow{(-1)^k v} \Omega^{k+n_2}(C_2)\xrightarrow{(-1)^k w} \Omega^{k+n_1}(C_1)
\]
is a triangle in $\cC$.

\begin{Theorem}[\cite{KV87}]\label{equivalence of singularity category}
The following hold:
\begin{enumerate}
\item\label{equivalence of singularity category:1} If $\cC$ is a left triangulated category, then $\bZ\cC$ becomes a triangulated category with suspension functor $\Sigma$ and with triangles being the standard triangles given above.
\item\label{equivalence of singularity category:2} Let $\cE$ be an exact category with enough projectives $\cP$. Then there exists an equivalence of triangulated categories
\[
\bZ\underline{\cE}\xrightarrow{\cong} K^{-,b}(\cP)/K^b(\cP)
\]
sending an object $(\underline{E},n)$ to a complex $P^{\bullet}[-n]$, where $P^{\bullet}$ is a projective resolution of $E$ concentrated in degrees $\leq 0$.
\end{enumerate}
\end{Theorem}
Here $K^b(\cP)$ and $K^{-,b}(\cP)$ denote the bounded homotopy category with components in $\cP$ and the right bounded homotopy category with bounded homology and with components in $\cP$, respectively. The Verdier quotient 
\[
D_{\operatorname{sing}}(\cE):=K^{-,b}(\cP)/K^b(\cP)
\]
is the singularity category of $\cE$. We refer to Section 4 in \cite{Kra10} for details on Verdier quotients and localization of triangulated categories.

\section{Cluster tilting subcategories}\label{Section Cluster-tilting subcategories}
Let $\cE$ be an additive category and $\cM$ a full additive subcategory of $\cE$. We recall the following notions.
\begin{enumerate}
\item A morphism $M\to E$ in $\cE$ with $M\in \cM$ is called a \emphbf{right} $\cM$\emphbf{-approximation} if the induced morphism $\cE(M',M)\to \cE(M',E)$ is an epimorphism for all $M'\in \cM$; 
\item A morphism $E\to M$ in $\cE$ with $M\in \cM$ is called a \emphbf{left} $\cM$\emphbf{-approximation} if the induced morphism $\cE(M,M')\to \cE(E,M')$ is an epimorphism for all $M'\in \cM$;
\item $\cM$ is \emphbf{contravariantly finite} in $\cE$ if for all objects $E\in \cE$ there exists a right $\cM$-approximation $M\to E$;
\item $\cM$ is \emphbf{covariantly finite} in $\cE$ if for all objects $E\in \cE$ there exists a left $\cM$-approximation $E\to M$;
\item $\cM$ is \emphbf{functorially finite} in $\cE$ if it is contravariantly finite and covariantly finite in $\cE$.
\end{enumerate}
We recall the definition of $d$ and $d\bZ$-cluster tilting subcategories in the following. If $\cE$ is triangulated with suspension functor $\Sigma$, then by $\Ext^i_{\cE}(E,E')$ we mean the Hom-space $\Hom_{\cE}(E,\Sigma^i(E'))$.

\begin{Definition}\label{Definition cluster tilting subcategory}
Let $\cE$ be an exact or a triangulated category, let $\cM$ be a full subcategory of $\cE$, and let $d>0$ be a positive integer. We say that $\cM$ is $d$\emphbf{-cluster tilting} in $\cE$ if the following hold:
\begin{enumerate}
\item\label{Definition cluster tilting subcategory:1} $\cM$ is functorially finite in $\cE$;
\item\label{Definition cluster tilting subcategory:2} We have
\begin{align*}
\cM & = \{E\in \cE \mid \Ext^i_{\cE}(E,\cM)=0 \text{ for }1\leq i\leq d-1\} \\ 
& = \{E\in \cE \mid \Ext^i_{\cE}(\cM,E)=0 \text{ for }1\leq i\leq d-1\};
\end{align*}
\item\label{Definition cluster tilting subcategory:3} If $\cE$ is exact then $\cM$ is a generating and cogenerating subcategory of $\cE$, i.e. for any object $E\in \cE$ there exists an admissible monomorphism $E\to M$ with $M\in \cM$ and an admissible epimorphism $M'\to E$ with $M'\in \cM$.
\end{enumerate}
If furthermore $\Ext^i_{\cE}(\cM,\cM)=0$ for $i\notin d\bZ$, then we say that $\cM$ is $d\bZ$\emphbf{-cluster tilting} in $\cE$.
\end{Definition}

\begin{Remark}\label{on dZ-cluster tilting}
We have the following:
\begin{enumerate}
\item If $\cE$ is exact with enough projectives and injectives, then condition \ref{Definition cluster tilting subcategory:3} follows from conditions \ref{Definition cluster tilting subcategory:1} and \ref{Definition cluster tilting subcategory:2}. Therefore, \ref{Definition cluster tilting subcategory:3} is often not assumed.
\item\label{on dZ-cluster tilting:2} If $\cE$ is an exact category with enough projectives and $\cM$ is $d$-cluster tilting in $\cE$, then $\cM$ is $d\bZ$-cluster tilting in $\cE$ if and only if $\Omega^d(\underline{\cM})\subset \underline{\cM}$.
\item\label{on dZ-cluster tilting:3} If $\cE$ is triangulated with suspension functor $\Sigma$ and $\cM$ is $d$-cluster tilting in $\cE$, then the following conditions are equivalent:
\begin{enumerate}
\item $\cM$ is $d\bZ$-cluster tilting in $\cE$;
\item $\Sigma^d(\cM)\subset \cM$;
\item $\Sigma^{-d}(\cM)\subset \cM$.
\end{enumerate}
\end{enumerate}
\end{Remark}
We need following result later:

\begin{Proposition}[Proposition 2.2.2 in \cite{Iya07a}]\label{characterization of cluster tilting}
Let $\cE$ be an exact category and $\cM$ a full subcategory of $\cE$. Assume $\cM$ is closed under direct summands and satisfies $\Ext^i_{\cE}(\cM,\cM)=0$ for $1\leq i\leq d-1$. Assume furthermore that condition \ref{Definition cluster tilting subcategory:1} and \ref{Definition cluster tilting subcategory:3} in Definition \ref{Definition cluster tilting subcategory} holds for $\cM$. Then the following are equivalent for each $0\leq n\leq d-1$:
\begin{itemize}
\item[$(a)$]$\cM$ is a $d$-cluster tilting subcategory of $\cE$; 
\item[$(b_n)$] If $E\in \cE$ satisfies $\Ext^i_{\cE}(\cM,E)=0$ for $1\leq i\leq n$, then there exists an exact sequence $0\to M_{d-n}\to \cdots \to M_1\to E\to 0$ with $M_j\in \cM$ for $1\leq j\leq d-n$;
\item[$(c_n)$] If $E\in \cE$ satisfies $\Ext^i_{\cE}(E,\cM)=0$ for $1\leq i\leq n$, then there exists an exact sequence $0\to E\to M^1\to \cdots \to M^{d-n}\to 0$ with $M^j\in \cM$ for $1\leq j\leq d-n$.
\end{itemize}
\end{Proposition}

\begin{proof}
For any $E\in \cE$ choose a right $\cM$-approximation $f\colon M\to E$ and an admissible epimorphism $g\colon M'\to E$ with $M'\in \cM$. Then the morphism $\begin{bmatrix}f & g\end{bmatrix}\colon M\oplus M'\to E$ is a right $\cM$-approximation and an admissible epimorphism. Similarly, one can construct a left $\cM$-approximation which is an admissible monomorphism for any $E\in \cE$.  Using this, the proof of Proposition 2.2.2 in \cite{Iya07a} goes through in exactly the same way for exact categories, so the claim holds. 
\end{proof}

The following result shows that condition $(b_0)+(c_0)$ is equivalent to $d$-cluster tilting without the assumption that $\cM$ is functorially finite.

\begin{Proposition}\label{characterization of cluster tilting with functorially finite}
Let $\cE$ be an exact category and $\cM$ a full subcategory of $\cE$. Assume $\cM$ is closed under direct summands and satisfies $\Ext^i_{\cE}(\cM,\cM)=0$ for $1\leq i\leq d-1$. Then the following are equivalent:
\begin{enumerate}
\item\label{characterization of cluster tilting with functorially finite:1} $\cM$ is a $d$-cluster tilting subcategory of $\cE$; 
\item\label{characterization of cluster tilting with functorially finite:2} For each $E\in \cE$ there exists exact sequences 
\begin{align*}
& 0\to M_{d}\to \cdots \to M_1\to E\to 0 \\
& 0\to E\to M^1\to \cdots \to M^{d}\to 0
\end{align*} 
with $M^i\in \cM$ and $M_i\in \cM$ for $1\leq i\leq d$.
\end{enumerate}
\end{Proposition}

\begin{proof}
The implication \ref{characterization of cluster tilting with functorially finite:1}$\implies$\ref{characterization of cluster tilting with functorially finite:2} follows from the implications $(a){\implies}(b_0)$ and $(a){\implies}(c_0)$ in Proposition \ref{characterization of cluster tilting}. For the converse, let $E\in \cE$, and choose an exact sequence
\[
0\to E\xrightarrow{f^1}M^1\xrightarrow{f^2}\cdots \xrightarrow{f^d}M^d\to 0
\]
with $M^i\in \cM$ for $1\leq i\leq d$. Let $M\in \cM$ be arbitrary. Applying $\Hom_{\cE}(-,M)$ to the short exact sequence
\[
0\to \Coker f^i\to M^{i+1}\to \Coker f^{i+1}\to 0
\]
we get that 
\[
\Ext^{j+1}_{\cE}(\Coker f^{i+1},M)\cong \Ext^{j}_{\cE}(\Coker f^{i},M)
\]
for $1\leq j\leq d-2$. Hence, it follows that
\[
0=\Ext^{d-1}_{\cE}(M^d,M)\cong \Ext^{d-2}_{\cE}(\Coker f^{d-2},M)\cong \cdots \cong \Ext^1_{\cE}(\Coker f^1,M).
\]
Therefore, the map
\[
\Hom_{\cE}(M^1,M)\to \Hom_{\cE}(E,M)
\]
is an epimorphism. Since $M\in \cM$ was arbitrary, we get that $f^1\colon E\to M^1$ is a left $\cM$-approximation. Hence, $\cM$ is covariantly finite. Furthermore, if we assume $\Ext^i_{\cE}(M',E)=0$ for $1\leq i\leq d-1$ and $M'\in \cM$, then the same argument as above with $M$ replaced by $E$ shows that the map 
\[
\Hom_{\cE}(M^1,E)\to \Hom_{\cE}(E,E)
\]
is an epimorphism. Therefore, $f^1\colon E\to M^1$ is a split monomorphism. Since $f^1$ is also an admissible monomorphism, it follows that $E$ is a summand of $M^1$. Hence, we get that $E\in \cM$. Since  $\Ext^i_{\cE}(\cM,\cM)=0$ for $1\leq i\leq d-1$, this shows the equality 
\[
\cM=\{E\in \cE\mid \Ext^i_{\cE}(\cM,E)=0 \text{ for }1\leq i\leq d-1 \}.
\] 
 The fact that $\cM$ is contravariantly finite and the equality 
\[
\cM=\{E\in \cE\mid \Ext^i_{\cE}(E,\cM)=0 \text{ for }1\leq i\leq d-1 \}
\] 
follows by a dual argument. Finally, $\cM$ being generating and cogenerating follows from the existence of the exact sequences in part \ref{characterization of cluster tilting with functorially finite:2} of the proposition. 
\end{proof}

\section{Cluster tilting subcategories of left triangulated categories}\label{Section Cluster tilting subcategories of left triangulated categories}

In this section we define cluster tilting subcategories of left triangulated categories. We show that when the ambient category is triangulated, then this coincides with the classical definition. Finally, we show that if a left triangulated category $\cC$ has a $d\bZ$-cluster tilting subcategory, then so does the stabilization $\mathbb{Z}\cC$.

Let $\cC$ be a left triangulated category. We call a sequence in $\cC$
\[
\Omega^d(C_1)\xrightarrow{\alpha_{d+2}}C_{d+2}\xrightarrow{\alpha_{d+1}} C_{d+1}\xrightarrow{\alpha_{d}} \cdots \xrightarrow{\alpha_{2}} C_2\xrightarrow{\alpha_{1}} C_1
\]
a $(d+2)$\emphbf{-angle} if there exists a diagram
\begin{align*}
\xymatrix@!=.5pc{ & C_{d+1} \ar[rr]^{\alpha_{d}} \ar[rd] && C_{d} \ar[rd] && \cdots && C_{2} \ar[rd]^{\alpha_{1}} \\ C_{d+2} \ar[ru]^{\alpha_{d+1}} \ar@{<-{|-}}[rr] && C_{d.5} \ar[ru] \ar@{<-{|-}} [rr] && C_{d-1.5} & \cdots & C_{2.5} \ar[ru] \ar@{<-{|-}}[rr] && C_{1}} 
\end{align*}
where an arrow $\xymatrix@!=.5pc{ C'\ar@{<-{|-}}[r] & C}$ denotes a morphism $C'\leftarrow \Omega(C)$ in $\cC$, each oriented triangle is a triangle in $\cC$, each non-oriented triangle commute, and $\alpha_{d+2}$ is equal to the composite
\[
\Omega^d(C_{1})\to \Omega^{d-1}(C_{2.5})\to \cdots \to \Omega(C_{d.5})\to C_{d+2}.
\]
Note that here the symbol $n.5$ means $n+0.5$. This definition of $(d+2)$-angle differs slightly from \cite{IY08}, where they do not include the morphism $\Omega^d(C_1)\to C_{d+2}$ in the definition.

\begin{Definition}\label{d-cluster tilting in left triangulated}
Let $\cC$ be a left triangulated category and $d>0$ a positive integer. A full additive subcategory $\cX$ of $\cC$ is $d$\emphbf{-cluster tilting} if it satisfies the following:
\begin{enumerate}
\item\label{d-cluster tilting in left triangulated:1} $\cX$ is closed under direct summands in $\cC$;
\item\label{d-cluster tilting in left triangulated:2} For all objects $C$ in $\cC$ there exist $(d+2)$-angles
\[
0\to C\to X^1\to \cdots \to X^d\to 0
\]
and
\[
0\to X_d\to \cdots \to X_1\to C\to 0
\]
with $X_i,X^i\in \cX$ for $1\leq i\leq d$;
\item\label{d-cluster tilting in left triangulated:3} For $C$ in $\cC$ and $X$ in $\cX$  the map
\[
\Hom_{\cC}(\Omega^{i-1}(X),C)\to \Hom_{\cC}(\Omega^{i}(X),\Omega(C)) \quad f\mapsto \Omega(f)
\]
is an isomorphism for $1\leq i\leq d-1$;
\item\label{d-cluster tilting in left triangulated:4} If $X$ and $X'$ are in $\cX$, then
\[
\Hom_{\cC}(\Omega^i(X'),X)=0
\]
for $1\leq i\leq d-1$.
\end{enumerate}
If furthermore $\Omega^d(\cX)\subset \cX$, then we say that $\cX$ is $d\bZ$\emphbf{-cluster tilting} in $\cC$.
\end{Definition}
We use the terminology $d$-cluster tilting since for an exact category $\cE$ with enough projectives we then get a correspondence between $d$-cluster tilting subcategories of $\cE$ and $\underline{\cE}$, see Theorem \ref{correspondence cluster tilting exact and stable}.

\begin{Remark}
We see that condition \ref{d-cluster tilting in left triangulated:2} in Definition \ref{d-cluster tilting in left triangulated} is similar to condition \ref{characterization of cluster tilting with functorially finite:2} in Proposition \ref{characterization of cluster tilting with functorially finite}, and it can be considered as a substitute of  Definition \ref{Definition cluster tilting subcategory} \ref{Definition cluster tilting subcategory:2} and of functorially finiteness. The fact that it behaves much better under stabilization is also a crucial property we need.
\end{Remark}

Next we show that Definition \ref{d-cluster tilting in left triangulated} and Definition \ref{Definition cluster tilting subcategory} are equivalent when $\Omega$ is an automorphism, i.e. when $\cC$ is triangulated. In the following we let $\Sigma$ denote the quasi-inverse of $\Omega$.

\begin{Proposition}\label{equivalence of definitions}
A subcategory of a triangulated category $\cC$ is $d$ or $d\bZ$-cluster tilting in the sense of Definition \ref{d-cluster tilting in left triangulated} if and only if it is $d$ or $d\bZ$-cluster tilting in the sense of Definition \ref{Definition cluster tilting subcategory}.
\end{Proposition}

\begin{proof}
By Remark \ref{on dZ-cluster tilting} \ref{on dZ-cluster tilting:3} the claim for $d\bZ$-cluster tilting subcategories follows immediately from the claim for $d$-cluster tilting subcategories. Hence, we only prove the latter. Note that a $d$-cluster tilting subcategory in the sense of Definition \ref{Definition cluster tilting subcategory} obviously satisfies \ref{d-cluster tilting in left triangulated:1}, \ref{d-cluster tilting in left triangulated:3} and  \ref{d-cluster tilting in left triangulated:4} in Definition \ref{d-cluster tilting in left triangulated}, while axiom \ref{d-cluster tilting in left triangulated:2} follows from \cite[Corollary 3.3]{IY08} and its dual. For the converse, assume $\cX\subset \cC$ is $d$-cluster tilting in the sense of Definition \ref{d-cluster tilting in left triangulated}. As usual, for subcategories $\cY'$ and $\cY$ of $\cC$, we denote by $\cY'*\cY$ the subcategory consisting of all $C\in \cC$ admitting a triangle $ Y'\to C\to Y\to \Sigma Y'$ with $Y'\in \cY'$ and $Y\in \cY$. By Definition \ref{d-cluster tilting in left triangulated} \ref{d-cluster tilting in left triangulated:2} we have $\cC=\Omega^{d-1}\cX*\cdots*\Omega\cX*\cX$. For each $C\in \cC$, take a triangle 
\[
Y\xrightarrow{f} C\xrightarrow{g} X\to \Sigma Y 
\]
 with $Y\in \Omega^{d-1}\cX*\cdots*\Omega\cX$ and $X\in \cX$. Since $\Hom_{\cC}(Y,\cX)=0$ by Definition \ref{d-cluster tilting in left triangulated} \ref{d-cluster tilting in left triangulated:4}, $g$ is a left $\cX$-approximation and hence $\cX$ is covariantly finite. Moreover, if $\Hom_{\cC}(\Omega^iX,C)=0$ for $1\leq i\leq d-1$, then $f=0$ and hence $C\in \cX$. Since $\Hom_{\cC}(\cX,\Sigma^iX)=0$ for $1\leq i\leq d-1$ and $X\in \cX$ by Definition \ref{d-cluster tilting in left triangulated} \ref{d-cluster tilting in left triangulated:4}, this shows that
\[
\cX = \{C\in \cC\mid \Hom_{\cC}(\cX,\Sigma^iC)=0 \text{ for }1\leq i\leq d-1\}.
\]
 The fact that $\cX$ is contravariantly finite and the equality
\[
\cX = \{C\in \cC\mid \Hom_{\cC}(C,\Sigma^i(\cX))=0 \text{ for }1\leq i\leq d-1\}
\]
is shown dually.
\end{proof}
Recall that the stabilization $\bZ\cC$ of a left triangulated category $\cC$ is a triangulated category, see Theorem \ref{equivalence of singularity category}.

\begin{Definition}
Let $\cC$ be a left triangulated category and $\cX$ a $d\bZ$-cluster tilting subcategory of $\cC$. Define $d\bZ\cX$ to be the full subcategory of $\bZ\cC$ consisting of all objects isomorphic to objects of the form $(X,dk)$ with $X\in \cX$ and $k\in \bZ$.
\end{Definition}

 Our goal is to show that $d\bZ\cX$ is $d\bZ$-cluster tilting in $\bZ\cC$.

\begin{Lemma}\label{closed under direct summands}
The subcategory $d\bZ\cX$ is closed under direct summands. 
\end{Lemma}

\begin{proof}
Two objects $(C,n)$ and $(C',n')$ are isomorphic in $\bZ\cC$ if and only if there exists an integer $k$ such that $\Omega^{k+n}(C)$ and $\Omega^{k+n'}(C')$ are isomorphic in $\cC$. Hence, $d\bZ\cX$ consists of all objects $(C,n)$ such that there exists an integer $k$ with $\Omega^{dk+n}(C)\in \cX$. Now assume that 
\[
(C_1,n_1)\oplus (C_2,n_2)\in d\bZ\cX.
\]
Choose $n:=\min(n_1,n_2)$. Then $(C_1,n_1)\cong (\Omega^{n_1-n}(C_1),n)$ and $(C_2,n_2)\cong (\Omega^{n_2-n}(C_2),n)$, and hence
\[
(C_1,n_1)\oplus (C_2,n_2)\cong (\Omega^{n_1-n}(C_1)\oplus \Omega^{n_2-n}(C_2),n)\in d\bZ\cX.
\] 
Therefore there exists an integer $k$ such that 
\[
\Omega^{dk+n_1}(C_1)\oplus \Omega^{dk+n_2}(C_2)\in \cX.
\]
Since $\cX$ is closed under direct summands by Definition \ref{d-cluster tilting in left triangulated} \ref{d-cluster tilting in left triangulated:1}, we have that 
\[
\Omega^{dk+n_1}(C_1)\in \cX \quad \text{and} \quad \Omega^{dk+n_2}(C_2)\in \cX
\]
and hence
\[
(C_1,n_1)\in d\bZ\cX \quad \text{and} \quad (C_2,n_2)\in d\bZ\cX.
\]
This proves the claim.
\end{proof}

\begin{Lemma}\label{Hom is zero}
If $(X,dn)\in d\bZ\cX$ and $(X',dn')\in d\bZ\cX$, then
\[
\Hom_{\bZ\cC}(\Omega^i(X,dn),(X',dn'))=0
\]
for $1\leq i\leq d-1$
\end{Lemma}

\begin{proof}
We have that 
\[
\Hom_{\bZ\cC}(\Omega^i(X,dn),(X',dn')) = \colim_{k'} \Hom_{\cC}(\Omega^i(\Omega^{k'+dn}(X)),\Omega^{k'+dn'}(X')).
\]
But
\[
\Hom_{\cC}(\Omega^i(\Omega^{dk+dn}(X)),\Omega^{dk+dn'}(X'))=0
\]
for all $k$ such that $dk+dn>0$ and $dk+dn'>0$ by Definition \ref{d-cluster tilting in left triangulated} \ref{d-cluster tilting in left triangulated:4}, since $\Omega^{dk+dn}(X)\in \cX$ and $\Omega^{dk'+dn}(X')\in \cX$. Hence, the colimit must be $0$, which proves the claim.
\end{proof}

\begin{Theorem}\label{stablization d-cluster tilting}
Let $\cC$ be a left triangulated category and $\cX$ a $d\bZ$-cluster tilting subcategory of $\cC$. Then $d\bZ\cX$ is a $d\bZ$-cluster tilting subcategory of the triangulated category $\bZ\cC$.
\end{Theorem}

\begin{proof}
We show that $d\bZ\cX$ satisfies Definition \ref{d-cluster tilting in left triangulated}. Note that axiom \ref{d-cluster tilting in left triangulated:1} and \ref{d-cluster tilting in left triangulated:4} follows from Lemma \ref{closed under direct summands} and Lemma \ref{Hom is zero} respectively, and axiom \ref{d-cluster tilting in left triangulated:3} is clear since $\bZ\cC$ is a triangulated category. It therefore only remains to show axiom \ref{d-cluster tilting in left triangulated:2}. Let $(C,n)\in \bZ\cC$ be arbitrary. Choose $k$ such that $dk<n$. Then we have an isomorphism $(C,n)\cong (\Omega^{n-dk}(C),dk)$. Hence, we can assume for simplicity that $n=dk$. Now choose $(d+2)$-angles
\[
0\to C\to X^1\to \cdots X^d\to 0
\]
\[
0\to X_d\to \cdots \to X_1\to C\to 0
\]
in $\cC$ with $X_i,X^i\in \cX$ for $1\leq i\leq d$. Applying Axiom (T3) in Definition \ref{Left triangulated categories} repeatedly, we obtain $(d+2)$-angles
\[
0\to \Omega^{dk}(C)\to \Omega^{dk}(X^1)\to \cdots \Omega^{dk}(X^d)\to 0
\]
\[
0\to \Omega^{dk}(X_d)\to \cdots \to \Omega^{dk}(X_1)\to \Omega^{dk}(C)\to 0
\]
in $\cC$. But they give $(d+2)$-angles 
\[
0\to (C,dk)\to (X^1,dk)\to \cdots (X^d,dk)\to 0
\]
\[
0\to (X_d,dk)\to \cdots \to (X_1,dk)\to (C,dk)\to 0
\]
in $\bZ\cC$, which prove the claim.
\end{proof}

\begin{Remark}
Note that the proof of Theorem \ref{stablization d-cluster tilting} does not use axiom \ref{d-cluster tilting in left triangulated:3} in Definition \ref{d-cluster tilting in left triangulated}. This axiom is needed to prove Theorem \ref{correspondence cluster tilting exact and stable} (and in particular Lemma \ref{rigidity}) to get a correspondence between $d\bZ$-cluster tilting subcategories in the exact category $\cE$ and in the left triangulated category $\underline{\cE}$
\end{Remark}

The category $d\bZ\cX$ is also equivalent to the stabilization of $\cX$ with respect to the functor
\[
\Omega^d\colon \cX\to \cX.
\]
Hence, it can be computed without describing the categories $\cC$ and $\bZ\cC$, which is more complicated to do in general.

Since $d\bZ\cX$ is a $d\bZ$-cluster tilting subcategory of $\bZ\cC$, it has the structure of a $(d+2)$-angulated category \cite[Theorem 4.1]{GKO13}, where 
\[
\Sigma^d\colon d\bZ\cX\to d\bZ\cX \quad \Sigma^d(X,nd)=(X,n(d-1))
\]
is the suspension functor (see \eqref{suspension and loop}) applied $d$ times.  The $(d+2)$-angles in the sense of \cite{GKO13} are all $(d+2)$-angles in $\bZ\cC$ in our sense 
\[
\Omega^d(E_1)\xrightarrow{\alpha_{d+2}} E_{d+2}\xrightarrow{\alpha_{d+1}}E_{d+1}\xrightarrow{\alpha_{d}}\cdots \xrightarrow{\alpha_{1}}E_{1} 
\]
where $E_i\in d\bZ\cX$ for $1\leq i\leq d+2$. 

\begin{Lemma}\label{d+2 angulated structure of dZX}
A sequence 
\[
\Omega^d(X_1,dn_1)\to (X_{d+2},dn_{d+2})\xrightarrow{}(X_{d+1},dn_{d+1})\xrightarrow{} \cdots \xrightarrow{}(X_{1},dn_{1})
\]
in $d\bZ\cX$ with $X_i\in \cX$ for $1\leq i\leq d+2$ is a $(d+2)$-angle if and only if it is induced from a sequence in $\cC$ of the form
\[
\Omega^{d(k+n_1+1)}(X_1)\xrightarrow{u_{d+2}}\Omega^{d(k+n_{d+2})}(X_{d+2})\xrightarrow{u_{d+1}}\cdots \xrightarrow{u_1} \Omega^{d(k+n_1)}(X_1)
\]
where
\[
\resizebox{\textwidth}{!}{$\Omega^{d(k+n_1+1)}(X_1)\xrightarrow{(-1)^{dk}u_{d+2}}\Omega^{d(k+n_{d+2})}(X_{d+2})\xrightarrow{(-1)^{dk}u_{d+1}}\cdots \xrightarrow{(-1)^{dk}u_1} \Omega^{d(k+n_1)}(X_1)$}
\]
is a $(d+2)$-angle in $\cC$.
\end{Lemma}

\begin{proof}
This follows immediately from the description of the triangles in $\bZ\cC$ together with axiom (T3) in Definition \ref{Left triangulated categories}.
\end{proof}

\section{\texorpdfstring{$d$}{}-cluster tilting in stable categories}\label{Section d-cluster tilting in stable categories}
Let $\cE$ be an exact category with enough projectives. In this section we compare cluster tilting subcategories in the exact category $\cE$ and those in the left triangulated category $\underline{\cE}$. Our main goal is to prove the following theorem:

\begin{Theorem}\label{correspondence cluster tilting exact and stable}
Let $\cE$ be an exact category with enough projectives $\cP$, and $\cM$ a full subcategory of $\cE$.
\begin{enumerate}
\item\label{correspondence cluster tilting exact and stable:1} $\cM$ is a $d$-cluster tilting subcategory of $\cE$ if and only if $\underline{\cM}$ is a $d$-cluster tilting subcategory of $\underline{\cE}$;
\item\label{correspondence cluster tilting exact and stable:2} $\cM$ is a $d\bZ$-cluster tilting subcategory of $\cE$ if and only if $\underline{\cM}$ is a $d\bZ$-cluster tilting subcategory of $\underline{\cE}$.
\end{enumerate}
\end{Theorem}

In the special case when $\cE$ is Frobenius (and hence $\underline{\cE}$ is triangulated) the theorem is easy and well-known.

We start by proving one direction of the theorem.

\begin{proof}[Proof of ``only if" part of Theorem \ref{correspondence cluster tilting exact and stable}]
If $\cM$ is $d\bZ$-cluster tilting, then for each $1\leq i\leq d-1$, we have $\Ext^d_{\cE}(\Omega^dM',M)=\Ext^{i+d}_{\cE}(M',M)=0$ where $M,M'\in \cM$. Therefore $\Omega^d(\underline{\cM})\subseteq \underline{\cM}$. Hence the ``only if" part of Theorem \ref{correspondence cluster tilting exact and stable} \ref{correspondence cluster tilting exact and stable:2} follows from the ``only if" part of Theorem \ref{correspondence cluster tilting exact and stable} \ref{correspondence cluster tilting exact and stable:1} by Remark \ref{on dZ-cluster tilting} \ref{on dZ-cluster tilting:2}.  

Assume $\cM$ is $d$-cluster tilting in $\cE$. Since $\cM$ is closed under direct summands, it follows using the description of the isomorphisms from \eqref{isomorphism in stable cat} that $\underline{\cM}$ is closed under direct summands. Now for $E\in \cE$ we can choose exact sequences
\begin{align*}
& 0\to E\xrightarrow{f^1}M^1\xrightarrow{f^2}\cdots \xrightarrow{f^d}M^d\to 0 \\
& 0\to M_d\xrightarrow{g_d} \cdots \xrightarrow{g_2} M_1\xrightarrow{g_1} E\to 0
\end{align*}
in $\cE$ with $M_i,M^i\in \cM$ for $1\leq i\leq d$ by Proposition \ref{characterization of cluster tilting}. By definition of the left triangulated structure of $\underline{\cE}$, we get $(d+2)$-angles
\begin{align*}
\xymatrix@!=.5pc{ & \underline{M^1} \ar[rr]^{\underline{f^2}} \ar[rd] && \underline{M^2}\ar[rd] && \cdots && \underline{M^{d-1}} \ar[rd]^{}\ar[rr]^{\underline{f^d}} && \underline{M^{d-1}} \ar[rd]\\ 
\underline{E} \ar[ru]^{\underline{f^1}} \ar@{<-{|-}}[rr] && \underline{\Coker f^1} \ar[ru]\ar@{<-{|-}}[rr] &&\underline{\Coker f^2}&&\cdots&& \underline{M^d}\ar[ru]^{1} \ar@{<-{|-}}[rr] && 0 } 
\end{align*}
and
\begin{align*}
\xymatrix@!=.5pc{ &  \underline{M_{d-1}} \ar[rr]^{\underline{f_{d-1}}} \ar[rd] && \underline{M_{d-2}}  && \cdots && \underline{M_{1}}\ar[rr]^{\underline{f_1}} \ar[rd]^{} && \underline{E} \ar[rd]^{} \\ 
\underline{M_{d}} \ar@{<-{|-}}[rr] \ar[ru]^{\underline{f_{d}}} && \underline{M_{d-2.5}} \ar[ru]^{} && \cdots  && \underline{\Ker f_1}\ar@{<-{|-}}[rr]\ar[ru] && \underline{E} \ar[ru]^{1} \ar@{<-{|-}}[rr] && 0 .} 
\end{align*}
This shows that part \ref{d-cluster tilting in left triangulated:2} of Definition \ref{d-cluster tilting in left triangulated} holds for $\underline{\cM}$. Also, by Lemma \ref{Syzygy functor full and faithful} the map 
\[
\Hom_{\underline{\cE}}(\Omega^{i-1}(\underline{M}),\underline{E})\to \Hom_{\underline{\cE}}(\Omega^{i}(\underline{M}),\Omega(\underline{E})) \quad \underline{f} \mapsto \Omega(\underline{f})
\]
is an isomorphism for any $M\in \cM$, $E\in \cE$ and $1\leq i\leq d-1$ since
\[
\Ext^1_{\cE}(\Omega^{i-1}M,P)\cong \Ext^i_{\cE}(M,P)=0
\]
for all $P\in \cP$. Hence, part \ref{d-cluster tilting in left triangulated:3} of Definition \ref{d-cluster tilting in left triangulated} also holds for $\underline{\cM}$. Finally, to prove part \ref{d-cluster tilting in left triangulated:4}, we use the basic fact that if $\Ext^1_{\cE}(E,\cP)=0$, then 
\[
\Hom_{\underline{\cE}}(\Omega(\underline{E}),\underline{E'})\cong \Ext^1_{\cE}(E,E')
\]
for all $E'\in \cE$. In particular, since $\Ext^{i}_{\cE}(M,\cP)=0$ for all $1\leq i \leq d-1$ and $M\in \cM$, it follows that
\[
\Hom_{\underline{\cE}}(\Omega^{i}(\underline{M}),\underline{M'})\cong \Ext^1_{\cE}(\Omega^{i-1}M,M')\cong \Ext^i_{\cE}(M,M')=0
\]
for $1\leq i\leq d-1$ and $M'\in \cM$. This proves the claim.
\end{proof}

\begin{proof}[Proof of Theorem \ref{Main theorem intro}]
By Theorem \ref{stablization d-cluster tilting} and the ``only if" part of Theorem \ref{correspondence cluster tilting exact and stable} we know that $d\bZ\underline{\cM}$ is $d\bZ$-cluster tilting in $\bZ\underline{\cE}$. Furthermore, under the equivalence $\bZ\underline{\cE}\xrightarrow{\cong} K^{-,b}(\cP)/K^b(\cP)$ in Theorem \ref{equivalence of singularity category} an object $(\underline{E},n)$ gets sent to $P^{\bullet}[-n]$ where $P^{\bullet}$ is a projective resolution of $E$. Then $\underline{Z^{-di}(P^{\bullet}[-n])}=\Omega^{di+n}(\underline{E})$ holds for $i\gg  0$. Thus $\underline{Z^{-di}(P^{\bullet}[-n])}\in \underline{\cM}$ holds for all $i\gg 0$ exactly when $(\underline{E},n)\in d\bZ\underline{\cM}$. Hence, $d\bZ\underline{\cM}$ corresponds to the subcategory in Theorem \ref{Main theorem intro} under the equivalence $\bZ\underline{\cE}\xrightarrow{\cong} K^{-,b}(\cP)/K^b(\cP)$, which proves the claim.
\end{proof}

\begin{Remark}
Combining Lemma \ref{d+2 angulated structure of dZX} with the description of the triangles in $\underline{\cE}$, one obtains a description of the $(d+2)$-angles in $d\bZ\underline{\cM}$. Explicitly, consider all sequences
\[
\Omega^d(\underline{M_1},dn_1)\to (\underline{M}_{d+2},dn_{d+2})\xrightarrow{} \cdots \xrightarrow{}(\underline{M}_{1},dn_{1}) 
\]
in $d\bZ\underline{\cE}$ with $\underline{M_i}\in \underline{\cM}$ for all $i$, which arises from a sequence in $\cE$
\[
N_1'\xrightarrow{u_{d+2}} N_{d+2}\xrightarrow{u_{d+1}} \cdots \xrightarrow{u_{1}}N_1
\]
under isomorphisms $\underline{N_i}\cong \Omega^{d(k+n_{i})}\underline{M_{i}}$ and $\underline{N'_{1}}\cong \Omega^{d(k+n_{1}+1)}\underline{M_{1}}$, and such that there exists a commutative diagram
\begin{align}\label{Standard d+2-angle}
\xymatrix@C=3em{
0\ar[r]& N_{d+2} \ar[r]^-{(-1)^{dk}u_{d+1}} & N_{d+1} \ar[r]^-{(-1)^{dk}u_{d}} & \cdots \ar[r]^{(-1)^{dk}u_2} & N_2\ar[r]^-{(-1)^{dk}u_1}& N_{1} \ar[r] & 0  \\
0\ar[r] & N'_{1} \ar[r]^{}\ar[u]^{(-1)^{dk}u_{d+2}} & P_d \ar[r]^{}\ar[u] &  \cdots \ar[r] & P_1 \ar[r] \ar[u] & N_{1} \ar[r] \ar[u]^1 & 0
}
\end{align}
with exact rows, where $P_1,\cdots, P_d$ are projective. If we call such a sequence in $d\bZ\underline{\cM}$ a standard $(d+2)$-angle, then a $(d+2)$-angle of $d\bZ\underline{\cM}$ is precisely a sequence which is isomorphic to a standard $(d+2)$-angle. 
\end{Remark}

We now want to show the ``if" part of Theorem \ref{correspondence cluster tilting exact and stable}. To this end, we fix a full subcategory $\cM$ of $\cE$ and assume $\underline{\cM}$ is $d$-cluster tilting in the left triangulated category $\underline{\cE}$. Our goal is to show that $\cM$ is $d$-cluster tilting in $\cE$.

\begin{Lemma}\label{rigidity}
If $M,M'\in \cM$, then
\[
\Ext^i_{\cE}(M,M')=0
\]
for $1\leq i\leq d-1$.
\end{Lemma}

\begin{proof}
By Lemma \ref{Syzygy functor full and faithful} and Definition \ref{d-cluster tilting in left triangulated} \ref{d-cluster tilting in left triangulated:3} we have that 
\[
\Ext^i_{\cE}(M,P)\cong \Ext^1_{\cE}(\Omega^{i-1}M,P)=0
\]
for $P\in \cE$ projective and $1\leq i\leq d-1$. Next, using the fact that if $\Ext^1_{\cE}(E,\cP)=0$, then $ \Ext^1_{\cE}(E,E')\cong \Hom_{\underline{\cE}}(\Omega(\underline{E}),\underline{E'})$ for all $E'\in \cE$, it follows that 
\[
\Ext^i_{\cE}(M,M')\cong \Ext^1_{\cE}(\Omega^{i-1}M,M')\cong \Hom_{\underline{\cE}}(\Omega^i(\underline{M}),\underline{M'})=0
\]
for all $1\leq i\leq d-1$ by Definition \ref{d-cluster tilting in left triangulated} \ref{d-cluster tilting in left triangulated:4}. This proves the claim.
\end{proof}

\begin{Lemma}\label{resolving and coresolving by objects in M}
For any object $E\in \cE$ there exist exact sequences
\[
0\to E\to M^1\to \cdots \to M^d\to 0
\]
and
\[
0\to M_d\to \cdots \to M_1\to E\to 0
\]
where $M_i,M^i\in \cM$ for $1\leq i\leq d$.
\end{Lemma}

\begin{proof}
Let $E\in \cE$ be arbitrary, and choose a $(d+2)$-angle
\[
0\to E\to X^1\to \cdots \to X^d\to 0
\]
in $\underline{\cE}$ with $X^i\in \underline{\cM}$ for $1\leq i\leq d$. Hence, by definition of triangles in $\underline{\cE}$ (see Section \ref{Section Left triangulated categories}) there exist a projective object $P\in \cE$ and an exact sequence
\[
0\to E\oplus P\xrightarrow{f^1}N^1\xrightarrow{f^2}\cdots \xrightarrow{f^d}N^d\to 0.
\]
in $\cE$, with $N^i\in \cM$ for $1\leq i\leq d$. Since
\[
0=\Ext^{d-1}_{\cE}(N^d,P)\cong \Ext^{d-2}_{\cE}(\Coker f^{d-2},P)\cong \cdots \cong \Ext^1_{\cE}(\Coker f^1,P)
\]
it follows that the map
\[
\Hom_{\cE}(N^1,P)\to \Hom_{\cE}(E\oplus P,P)\to 0
\]
is an epimorphism. Hence, the inclusion
\[
P\xrightarrow{\begin{bmatrix}0\\1\end{bmatrix}}E\oplus P\xrightarrow{f^1}N^1
\]
is a split monomorphism. The inclusion is also a composite of two admissible monomorphism, and it is therefore admissible. Therefore, its cokernel exists, which we denote by $M^1$. We can therefore write the sequence as
\[
0\to E\oplus P\xrightarrow{\begin{bmatrix}g^1&0\\0&1\end{bmatrix}}M^1\oplus P\xrightarrow{\begin{bmatrix}g^2&0\end{bmatrix}}N^2\xrightarrow{f^3} \cdots \xrightarrow{f^d}N^d\to 0
\]
for some morphisms $g^1,g^2$. It follows from \cite[Corollary 2.18]{Bue10} that the sequence
\[
0\to E\xrightarrow{g^1}M^1\xrightarrow{g^2}N^2\xrightarrow{f^3}\cdots \xrightarrow{f^d}N^d\to 0
\]
is exact. This proves one part of the lemma. The other part is proved dually.
\end{proof}

\begin{proof}[Proof of ``if" part of Theorem \ref{correspondence cluster tilting exact and stable}]
Part \ref{correspondence cluster tilting exact and stable:1} follows from Proposition \ref{characterization of cluster tilting with functorially finite}, \\ Lemma \ref{rigidity} and Lemma \ref{resolving and coresolving by objects in M}. Part \ref{correspondence cluster tilting exact and stable:2} follows from part \ref{correspondence cluster tilting exact and stable:1} together with Remark \ref{on dZ-cluster tilting} \ref{on dZ-cluster tilting:2}.
\end{proof}

\section{Gorenstein projectives}\label{Section Gorenstein projectives}
In this section we consider the subcategory of Gorenstein projective objects in $\cE$. These objects were first introduced in \cite{AB69} for modules over a noetherian ring. We refer to \cite{Che10} for a survey of the theory for Artin algebras, and to \cite{EJ00} for more general rings. 

Let $\cE$ be an exact category with enough projectives $\cP$. Recall that a long exact sequence
\[
P^{\bullet}=\cdots \to P^0\to P^{1}\to \cdots
\]
of projective objects in $\cE$ is called \emphbf{totally acyclic} if the complex
\[
\Hom_{\cE}(P^{\bullet},Q)= \cdots\to \Hom_{\cE}(P^{1},Q)\to \Hom_{\cE}(P^0,Q)\to \cdots
\] 
is acyclic for all projective objects $Q$ in $\cE$. An object $G\in \cE$ is called \emphbf{Gorenstein projective} if there exists a totally acyclic complex $P^{\bullet}$ with 
\[
G=Z^0(P^{\bullet}):=\Ker (P^0\to P^{1}).
\]
We let $\cG\cP(\cE)$ denote the subcategory of $\cE$ consisting of all Gorenstein projective objects. The subcategory $\cG\cP(\cE)$ is closed in $\cE$ under extensions, direct summands, and kernels of admissible epimorphisms. In fact, the proof of \cite[Proposition 2.1.7]{Che10} also works for exact categories. In particular, $\Omega\colon \underline{\cE}\to \underline{\cE}$ restricts to a functor 
\[
\Omega\colon \underline{\cG\cP}(\cE)\to \underline{\cG\cP}(\cE).
\]
The \emphbf{Gorenstein projective dimension} of an object $E\in \cE$, denoted $\dim_{\cG\cP(\cE)}E$, is the smallest integer $n$ such that $\Omega^n(\underline{E})\in \underline{\cG\cP}(\cE)$. We write $\dim_{\cG\cP(\cE)}E=\infty$ if no such integer exists. 

Since $\cG\cP(\cE)$ is an extension closed subcategory of $\cE$, it inherits an exact structure making the inclusion 
\[
\cG\cP(\cE)\to \cE
\]
into an exact functor. Under this exact structure $\cG\cP(\cE)$ becomes a Frobenius exact category with projective/injective objects being the objects in $\cP$, see \cite[Proposition 2.1.11]{Che10}. Hence, $\underline{\cG\cP}(\cE)$ is a triangulated category. In particular, $\Omega\colon \underline{\cG\cP}(\cE)\to \underline{\cG\cP}(\cE)$ is an autoequivalence, and the quasi-inverse of $\Omega$ is the suspension functor for the triangulated category. The triangles in $\underline{\cG\cP}(\cE)$ are precisely all triangles in the left triangulated category $\underline{\cE}$ with components in $\underline{\cG\cP}(\cE)$. In particular, we see that the canonical functor 
\[
\underline{\cE}\to \bZ\underline{\cE} \quad \quad \underline{E}\mapsto (\underline{E},0)
\]
restrict to a functor of triangulated categories
\[
\underline{\cG\cP}(\cE)\to \bZ\underline{\cE}.
\]
This functor is fully faithful since $\Omega$ is an autoequivalence on $\underline{\cG\cP}(\cE)$. The result below gives necessary and sufficient condition for it to be an equivalence. It was first shown in \cite{Buc86} for a noetherian ring.

\begin{Proposition}\label{Buchweitz result}
The essential image of the functor $\underline{\cG\cP}(\cE)\to \bZ\underline{\cE}$ is
\[
\{(\underline{E},n)\in \bZ\underline{\cE}\mid \dim_{\cG\cP(\cE)}E<\infty\}.
\] 
In particular, the functor is an equivalence if and only if $\dim_{\cG\cP(\cE)}E<\infty$ for all $E\in \cE$.
\end{Proposition}  

\begin{proof}
If $(\underline{E},n)\cong (\underline{G},0)$ in $\bZ\underline{E}$ with $G\in \cG\cP(\cE)$, then $\Omega^{k+n}(\underline{E})\cong \Omega^k(\underline{G})$ in $\underline{\cE}$ for some $k>0$. Since $\Omega^k(\underline{G})\in \underline{\cG\cP}(\cE)$ it follows that $E$ has finite Gorenstein projective dimension. Conversely, if $\dim_{\cG\cP(\cE)}E=k<\infty$, then for any $n\in \bZ$ there exist isomorphisms
\[
(\underline{E},n)\cong (\Omega^{k}\underline{E},n-k)\cong (\Omega^{n-k}(\Omega^{k}\underline{E}),0)
\]
and since $\Omega^k\underline{E}\in \underline{\cG\cP}(\cE)$, it follows that $\Omega^{n-k}(\Omega^{k}\underline{E})\in \underline{\cG\cP}(\cE)$. This proves the claim.
\end{proof}

As a corollary of this result one can deduce \cite[Theorem 3.6]{BJO15} in the case of Artin algebras.

\begin{Proposition}
Let $\Lambda$ be an Artin algebra. Then the canonical functor
\[
\underline{\cG\cP}(\md \Lambda)\to D^b(\md \Lambda)/\perf \Lambda 
\]
sending $M$ to itself considered as a stalk complex in degree $0$ is an equivalence if and only if $\Lambda$ is Iwanaga-Gorenstein. 
\end{Proposition}

\begin{proof}
This follows from Theorem \ref{equivalence of singularity category} \ref{equivalence of singularity category:2} and Proposition \ref{Buchweitz result} applied to $\cE=\md \Lambda$ together with \cite[Proposition 4.2]{AR91a}.
\end{proof}

We now relate this to the theory of cluster tilting subcategories.

\begin{Theorem}\label{Cluster tilting inside Gorenstein projectives}
Assume $\dim_{\cG\cP(\cE)}E<\infty$ for all $E\in \cE$. Let $\cM$ be a $d\bZ$-cluster tilting subcategory of $\cE$. Then $\cM\cap \cG\cP(\cE)$ is a $d\bZ$-cluster tilting subcategory of $\cG\cP(\cE)$.
\end{Theorem}

\begin{proof}
Since $\cG\cP(\cE)$ is Frobenius and $\cM\cap \cG\cP(\cE)$ contains $\cP$, it follows that $\cM\cap \cG\cP(\cE)$ is a $d\bZ$-cluster tilting subcategory of $\cG\cP(\cE)$ if and only if $\underline{\cM\cap \cG\cP(\cE)}$ is a $d\bZ$-cluster tilting subcategory of the triangulated category $\underline{\cG\cP}(\cE)$. This is the easy case of Theorem \ref{correspondence cluster tilting exact and stable}. Now by Lemma \ref{Buchweitz result} we have an equivalence of triangulated categories 
\[
\underline{\cG\cP}(\cE)\xrightarrow{\cong} \bZ\underline{\cE}
\]
and since $d\bZ\underline{\cM}$ is a $d\bZ$-cluster tilting subcategory of $\bZ\underline{\cE}$ by Theorem \ref{stablization d-cluster tilting}, the preimage of $d\bZ\underline{\cM}$ is a $d\bZ$ cluster tilting subcategory of $\underline{\cG\cP}(\cE)$. Explicitly, the preimage consists of all objects $\underline{G}\in \underline{\cG\cP}(\cE)$ such that $\Omega^{dk}(\underline{G})\in \underline{\cM}$ for some integer $k\geq 0$. To show that this is equal to $\underline{\cM\cap \cG\cP(\cE)}$, we only need to show that such a $G$ is contained in $\cM$. Note first that since $G\in \cG\cP(\cE)$, it follows that $\Ext^i_{\cE}(G,P)=0$ for all $P\in \cP$. Hence, by a dimension shifting argument we get that
\[
\Ext^i_{\cE}(G,E)\cong \Ext^{i+j}_{\cE}(G,\Omega^j E)
\]
for any $E\in \cE$ and any $j\geq 0$. Now let $M\in \cM$ be arbitrary, and choose an integer $k\geq 0$ such that $\Omega^{dk}(\underline{G})\in \underline{\cM}$ and $\Omega^{dk}(\underline{M})\in \underline{\cG\cP}(\cE)$. It follows that
\[
\Ext^i_{\cE}(G,M)\cong \Ext^{i+dk}_{\cE}(G,\Omega^{dk}M)\cong \Ext^{i}_{\cE}(\Omega^{dk} G,\Omega^{dk}M)=0
\]
for $1\leq i\leq d-1$ where we use dimension shifting and the fact that $\Omega^{dk}M\in \cM$ since $\cM$ is $d\bZ$-cluster tilting. Since $M\in \cM$ was arbitrary, it follows that $G\in \cM$, which proves the claim.
\end{proof}

\section{Examples}\label{Section Example}

In this section we compute the singularity category for a higher Nakayama algebras of type $\bA^{\infty}_{\infty}$, see \cite{JK19}.  For higher Nakayama algebras of type $\tilde{\bA}$ we refer to Section 4.3 in \cite{McM18}.

\begin{Example}\label{Example}
We follow the notation in \cite{JK19}. Let $k$ be a field. Let $\underline{l}=(\cdots,l_{-1},l_0,l_1,\cdots)$ be the Kupisch Series of type $\bA^{\infty}_{\infty}$ \cite[Definition 3.10]{JK19} given by 
\begin{equation*}
l_i=
\begin{cases}
2, & \text{if}\ i \equiv 1 \operatorname{mod} 4 \\
3, & \text{otherwise}
\end{cases}
\end{equation*}
Let $A^{(2)}_{\underline{l}}$ be the $2$-Nakayama $k$-algebra with Kupisch series $\underline{l}$  \cite[Definition 3.13]{JK19}. Explicitly, $A^{(2)}_{\underline{l}}$ is the $k$-linear category given by the infinite periodic quiver
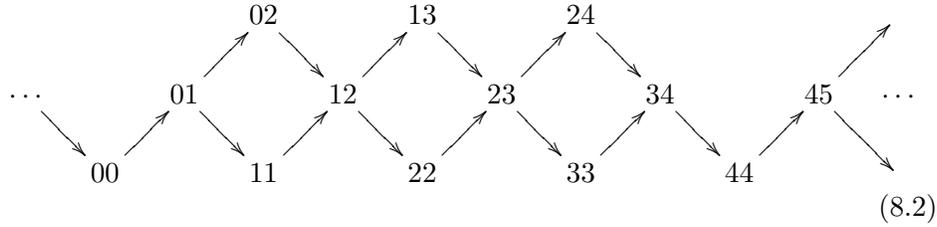
\begin{align}
\xymatrix@!=.5pc{ & & & 02 \ar[rd] && 13 \ar@{-->}[ll] \ar[rd] && 24 \ar@{-->}[ll] \ar[rd] & & & &  \\
\cdots \ar[rd] & &  01 \ar@{-->}[ll] \ar[ru]^{} \ar[rd] && 12 \ar@{-->}[ll] \ar[ru] \ar[rd] && 23 \ar[ru] \ar[rd] \ar@{-->}[ll] && 34 \ar[rd] \ar@{-->}[ll] && 45 \ar@{-->}[ll] \ar[rd] \ar[ru] & \cdots \\ 
& 00 \ar[ru]^{} && 11  \ar@{-->}[ll] \ar[ru] && 22 \ar@{-->}[ll] \ar[ru] && 33\ar[ru] \ar@{-->}[ll] && 44  \ar@{-->}[ll]\ar[ru] & & \ar@{-->}[ll]} 
\end{align}
with period $4$, and with relations making all squares commute, and such that the following composites are $0$.
\begin{align*}
& (a,a)\to (a,a+1)\to (a+1,a+1) & \text{for all }  a\in \bZ \\
& (a,a+2)\to (a+1,a+2)\to (a+1,a+3) & \text{if }  a\equiv 0,1 \operatorname{mod} 4 \\
& (a,a+1)\to (a+1,a+1)\to (a+1,a+2) & \text{if }  a\equiv 3 \operatorname{mod} 4.
\end{align*}
By \cite[Theorem 3.16]{JK19} we know that the category of finitely presented modules $\md A^{(2)}_{\underline{l}}$ has a $2\bZ$-cluster tilting subcategory $\cM^{(2)}_{\underline{l}}$. Explicitly, $\cM^{(2)}_{\underline{l}}$ is the $k$-linear additive category given by the infinite periodic quiver

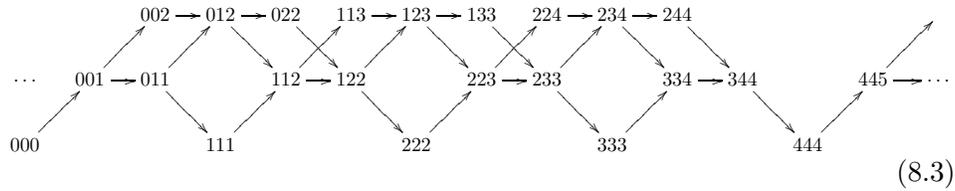
\begin{equation}\label{Example cluster tilting subcategory}
\resizebox{\textwidth}{!}{\xymatrix@!=1pc{ & & \color{red} 002\color{black}\ar[r] & \color{red}012\color{black}\ar[r]\ar[rd] & \color{red}022\color{black}\ar[rd] & \color{red}113\color{black} \ar[r]& \color{red}123\color{black}\ar[r]\ar[rd] & \color{red}133\color{black}\ar[rd] & \color{red}224\color{black}\ar[r] & \color{red}234\color{black}\ar[r]\ar[rd] & \color{red}244\color{black}\ar[rd] & & & &  \\
\cdots & \color{red}001\color{black} \ar[ru]  \ar[r] & \color{red}011\color{black} \ar[ru] \ar[rd] & & 112 \ar[ru]\ar[r] & 122\ar[ru] \ar[rd] & & 223\ar[ru] \ar[r] & 233\ar[ru] \ar[rd] & & 334 \ar[r] & 344 \ar[rd] & & \color{red}445\color{black} \ar[ru] \ar[r] & \cdots \\ 
 000 \ar[ru] &&& 111\ar[ru] &&& 222\ar[ru] &&& 333\ar[ru] &&& 444 \ar[ru] & & }} 
\end{equation}
with period $4$, and with relations
\[
\alpha_j(\underline{a}+e_i)\alpha_i(\underline{a})-\alpha_i(\underline{a}+e_j)\alpha_j(\underline{a})
\]
for any vertex $\underline{a}=(a_1,a_2,a_3)$ and integers $1\leq i,j\leq 3$, excluding the composites
\begin{align*}
012\to 112\to 113 \quad 022\to 122\to 123 \quad 123\to 223\to 224 \\
 133\to 233\to 234 \quad 344\to 444\to 445
\end{align*}
modulo $4$. Here $e_1=(1,0,0)$, $e_2=(0,1,0)$ and $e_3=(0,0,1)$ are the unit vectors and 
\[
\alpha_i(\underline{a})\colon \underline{a} \to \underline{a} + e_i  
\] 
is defined to be the unique arrow if it exists, and $0$ otherwise. In particular, some of the relations are zero relations, for example $000\to 001\to 011$. Also, similar as in Proposition 2.22 in \cite{JK19} the projective modules in $\md A^{(2)}_{\underline{l}}$ correspond to the vertices $(\lambda_1,\lambda_2,\lambda_3)$ with $\lambda_1=\lambda_3+1-l_{\lambda_3}$, which are precisely
\begin{multline*}
\cS:= \{(001),(011),(002),(012),(022),(113),(123), \\
(133),(224),(234),(244),(445) \operatorname{mod} 4\}
\end{multline*}
coloured in red in diagram \eqref{Example cluster tilting subcategory}.
Also, as obtained in the proof of Proposition 2.25 in \cite{JK19}, the $2$-syzygy of a non-projective object corresponding to a vertex $(\lambda_1,\lambda_2,\lambda_3)$ is
\[
\Omega^2(\lambda_1,\lambda_2,\lambda_3)= (\lambda_3+1-l_{\lambda_3},\lambda_1-1,\lambda_2-1).
\]
A simple computation shows that
\begin{align*}
& \Omega^2(444)= (233) \quad \Omega^2(233)=(112) \quad \Omega^2(112)=(000) \\
& \Omega^2(344)= (223) \quad \Omega^2(223)=(111) \quad \Omega^2(111)=(000) \\
& \Omega^2(334)= (222) \quad \Omega^2(222)=(011)\in \cS \\
& \Omega^2(333)= (122) \quad \Omega^2(122)=(001)\in \cS. 
\end{align*}
Hence, in $2\bZ\underline{\cM^{(2)}_{\underline{l}}}$ we have that any nonzero indecomposable object is isomorphic to $(000,d)$ for some integer $d\in \bZ$. Also, we have that 
\begin{equation*}
\Hom_{2\bZ\underline{\cM^{(2)}_{\underline{l}}}}((000,d),(000,d'))=
\begin{cases}
0, & \text{if}\ d\neq d' \\
k, & \text{if}\ d=d'.
\end{cases}
\end{equation*}
This shows that as additive categories we have an equivalence
\[
2\bZ\underline{\cM^{(2)}_{\underline{l}}}\cong \bigoplus_{i\in \bZ} \md k.
\]
\end{Example}

\begin{Example}\label{Example 2}
We continue with the notation of the previous example. Consider the canonical automorphism $\phi\colon A^{(2)}_{\underline{l}}\to A^{(2)}_{\underline{l}}$ sending a vertex $(i,j)$ to $(i-4,j-4)$. Let $\phi^*\colon \md A^{(2)}_{\underline{l}}\to \md A^{(2)}_{\underline{l}}$ be the induced equivalence, which sends a module $M$ to $M\circ \phi^{-1}$. Then $\phi^*$ restricts to an equivalence 
\[
\phi^*\colon \cM^{(2)}_{\underline{l}}\to \cM^{(2)}_{\underline{l}} 
\]
sending a vertex $(a_1,a_2,a_3)$ to $(a_1-4,a_2-4,a_3-4)$. This follows from the definition of the embedding $\cM^{(2)}_{\underline{l}}\to \md A^{(2)}_{\underline{l}}$, see \cite[Proposition 1.12]{JK19}. Let
\[
\widetilde{A}_{\underline{l'}}^{(2)}=A^{(2)}_{\underline{l}}/\phi
\] 
be the orbit category. This is a higher Nakayama algebra of type $\widetilde{\bA}$ with Kupisch Series $\underline{l'}=(3,2,3,3)$, see \cite[Definition 4.11]{JK19}. The canonical functor $F\colon A^{(2)}_{\underline{l}}\to A^{(2)}_{\underline{l}}/\phi$ induces a restriction functor
\[
F^*\colon \Md \widetilde{A}_{\underline{l'}}^{(2)}\to \Md A^{(2)}_{\underline{l}}
\]
with left adjoint (called the push down functor)
\[
F_*\colon \Md A^{(2)}_{\underline{l}}\to \Md \widetilde{A}_{\underline{l'}}^{(2)}. 
\] 
The functor $F_*$ is exact and restricts to a well-defined functor between the categories of finitely presented modules
\[
F_*\colon \md A^{(2)}_{\underline{l}}\to \md \widetilde{A}_{\underline{l'}}^{(2)}
\] 
see \cite[Section 14.3]{GR92}. Now the subcategory $\cM^{(2)}_{\underline{l'}}:= F_*(\cM^{(2)}_{\underline{l}})$ is $2\bZ$-cluster tilting in $\md \widetilde{A}_{\underline{l'}}^{(2)}$ by \cite[Theorem 4.12]{JK19}. It follows from \cite[Section 14.4]{GR92} that there exists an equivalence $\cM^{(2)}_{\underline{l'}}\cong \cM^{(2)}_{\underline{l}}/\phi^*$ such that the composite 
\[
\cM^{(2)}_{\underline{l}}\xrightarrow{F_*} \cM^{(2)}_{\underline{l'}}\xrightarrow{\cong} \cM^{(2)}_{\underline{l}}/\phi^*
\]
is just the natural projection functor. Now $F_*$ is exact and preserves projectives, and hence the computations of the projectives and the syzygies in Example \ref{Example} are also valid for $\cM^{(2)}_{\underline{l'}}$. It follows that as additive categories we have an equivalence
\[
2\bZ\underline{\cM^{(2)}_{\underline{l'}}}\cong \bigoplus_{1\leq i\leq 4} \md k.
\]
Finally, note that the module corresponding to vertex $(344)$ is injective in $\md \widetilde{A}_{\underline{l'}}^{(2)}$, and by the computations of the syzygy above we see that it has infinite projective dimension. This shows that algebra $\widetilde{A}_{\underline{l'}}^{(2)}$ is not Iwanaga-Gorenstein.
\end{Example}

\bibliography{Mybibtex}
\bibliographystyle{plain} 

\end{document}